\documentclass[11pt, reqno]{amsart}
\usepackage[utf8]{inputenc}
\usepackage{amsthm}
\usepackage{amsmath}
\usepackage{amssymb}
\usepackage{comment}
\usepackage[margin=1.5in]{geometry}
\usepackage{comment}
\usepackage[dvipsnames]{xcolor}

\usepackage{xcolor}

\usepackage[colorlinks=true, pdfstartview=FitV, linkcolor=blue, citecolor=blue, urlcolor=blue]{hyperref}

\providecommand{\wt}{\widetilde}

\providecommand{\C}{\mathbb C}

\providecommand{\N}{\mathbb N}
\renewcommand{\H}{\mathbb{H}}
\providecommand{\R}{\mathbb R}
\providecommand{\Z}{\mathbb Z}

\newcommand{\norm}[1]{\left|\left| #1 \right|\right|}

\newtheorem{theorem}{Theorem}[section]
\newtheorem{definition}[theorem]{Definition}
\newtheorem{lemma}[theorem]{Lemma}
\newtheorem{prop}[theorem]{Proposition}
\newtheorem{corollary}{Corollary}[theorem]
\newtheorem*{remark}{Remark}

\numberwithin{equation}{section}

\title{Additive Stability of Frames}
\author{Oleg Asipchuk}
\address{Oleg Asipchuk,  Florida International University,
	Department of Mathematics and Statistics,
	Miami, FL 33199, USA}
\email{aasip001@fiu.edu}
\author{Jacob Glidewell}
\address{Jacob Glidewell,  The University of Alabama,
	Department of Mathematics,
	Tuscaloosa, AL 35487-0350, USA}
\email{jbglidewell@crimson.ua.edu}
\author{Luis Rodriguez}
\address{Luis Rodriguez,  Florida International University,
	Department of Mathematics and Statistics,
	Miami, FL 33199, USA}
\email{lrodr073@fiu.edu}


\begin{document}
	
\begin{abstract}
	Given a frame in a finite dimensional Hilbert space we construct additive perturbations which decrease the condition number of the frame.  By iterating  this perturbation, we introduce an algorithm that produces a tight frame in a finite number of steps. Additionally, we give sharp bounds on additive perturbations which preserve frames and we study the effect of appending and erasing vectors to a given tight frame. We also discuss under which conditions our finite-dimensional results are extendable to infinite-dimensional Hilbert spaces.
\end{abstract}

\maketitle

\section{Introduction}

Hilbert space frames are powerful mathematical tools that provide a natural extension to the concept of bases. This subject
originated in the work of Duffin and 
Schaeffer \cite{Duffin} and has been greatly studied since the seminal work of  Daubechies, Grossmann, and Meyer \cite{DGM1986}.

\begin{definition}
	A finite or countable family of vectors ${\mathcal F} =\{x_j\}_{j\in J}$ in a Hilbert space $\mathbb{H} $ is a {\bf frame} if there are constants $0<A\le B<\infty$ satisfying:
	\[ A\|x\|^2 \le \sum_{j \in J}|\langle x,x_j\rangle|^2 \le B\|x\|^2,\mbox{ for all }x\in \mathbb{H} .\]
\end{definition}

If $A=B$ this is an {\it A-tight frame} and if $A=B=1$ this is a {\it Parseval frame}.  The largest $A$ and smallest $B$ satisfying
this inequality are called the optimal  {\it lower} (respectively, {\it upper}) {\it frame bounds}.
  The frame is said to be {\it equal-norm} if all elements have the same norm. 

\medskip
Typically, frames are not linearly independent, but linear independence often does not align well with the demands of applied problems; instead, the redundancy makes frames ideal for handling signals and other types of data, as it allows multiple representations of vectors in terms of the frame vectors. 
 
\medskip
Frames are stable, in the sense that a small perturbation  of a frame is still a frame.  This property   allows for accurate and reliable signal recovery even in the presence of noise or other disturbances.  Thus, quantitative estimates of the stability of frames under perturbations are crucial for applications such as signal reconstruction and de-noising.  
 
\medskip
Three  questions naturally arise with regard to the stability of frames:
\begin{enumerate}
    \item How much can a frame be perturbed and still remain a frame?

    \item How much does a frame need to be perturbed to become a tight frame?
    
\item  How does a frame change when more vectors are added?
\end{enumerate}
\noindent The goal of this paper is to give new progress on these questions in applicable ways.

Frame perturbations have been studied in many contexts, most notably with specific applications to Gabor frames,   \cite{Christensen1996,DV2018,SZ2001,SZ2003}, and exponential frames and bases;
%
the Kadec $\frac{1}{4}$-theorem \cite{Kadec}   states that $\{e^{2\pi i\lambda_n x}\}_{n\in \Z}$ is a   basis in $L^2[0,1]$ if $\sup_n|n-\lambda_n|<\frac{1}{4}$. O. Christensen extended this result to frames.  
 See   \cite{ChristensenGabor}  and also   \cite{BR1997,SZ1999}.  

We focus our study on the stability of frames under additive perturbations.  That is, given a frame $\{f_j\}_{  {j\in J}}$ in a Hilbert space ${\mathbb{H}}$ and $\epsilon>0$, we consider  the set $\{f_j+\delta_j\}_ {j\in J}$ where $||\delta_j||<\epsilon$ for each $j$. 
We can think of the   $\delta_j$  as "noise" that is added to the frame.

To our knowledge, no explicit results are available for additive perturbations of bases or frames in finite-dimensional Hilbert spaces. 
 
The Krein-Milman-Rutman theorem in \cite{Young} states that for a basis $\{x_n\}_{n\in\N}$ in a Banach space $X$, there exists $\epsilon_n>0$ for each $n$ such that if $\{y_n\}_{n \in \N}$ is a sequence in $X$ with $\norm{x_n-y_n}<\epsilon_n$, then $\{y_n\}_{n \in \N}$ is also a basis in $X$.  
  If $X$ has infinite dimension,  it follows from   the proof of  the theorem in \cite{Young} that the lower bound of the sequence $\{ \epsilon_n \}_{n\in\N}$ is $\epsilon=0$.  If $X$ is finite-dimensional, we estimate   $\epsilon$ in Proposition \ref{thm: StabilitywithPW} and Corollary \ref{cor: UTFstability}.

\subsection{Our results}

We have two theorems that we can call main results. Both of them solve Problem (2) but under different conditions. First, we consider the case when we have an upper bound on the norm of perturbation. 
 
\begin{theorem}\label{thm: betterFrame}
    Suppose $\{v_j\}_{j=1}^k$ is a frame in $\R^n$ with frame bounds $A<B$. For every $\epsilon >0$  there exists $\{\delta_j\}_{j=1}^k$ where  $||\delta_j||<\epsilon$ such that $\{\wt{v}_j\}=\{v_j+\delta_j\}_{j=1}^k$ is a frame with frame constants $0<A_1\le B_1$ where $$\frac{B_1}{A_1}<\frac{B}{A}.$$ 
\end{theorem}

	The {\it condition number} of a frame with optimal frame constants $A$ and $B$ is $\kappa= \frac BA$. The condition number indicates how "close" a frame is to being tight.  It also measures how sensitive the frame operator is to changes or errors in the input, and how much error in the output results from an error in the input. The problem of finding optimal  frame bounds and condition number for Gabor frames is  well studied. See for example the recent  \cite{Faulhuber2018,FS2023}. In applications it is desirable to have frames with a condition number as close to $1$ as possible; in the proof of Theorem \ref{thm: betterFrame} we provide an explicit expression for perturbations $\delta_j$ that reduce the condition number of a  given frame.

  Our next result shows that 
we can obtain a tight frame with a finite number of   perturbations 
\begin{theorem}\label{thm: tightAlg}
    Suppose $\{v_j\}_{j=1}^k$ is a  frame in $\R^n$ with optimal frame bounds $A< B$. We can construct $\{\delta_j\}_{j=1}^k$ such that $\{\wt{v}_j\}=\{v_j+\delta_j\}_{j=1}^k$ is a nonzero  tight frame, with an algorithm in at most $n-1$ steps.
\end{theorem}

  The algorithm used in the proof of Theorem \ref{thm: tightAlg} produces the explicit perturbations $\delta_j$   that make the frame tight.

In Theorem \ref{thm: betterFrame} the frame constants do not have to be optimal, but the optimality of the frame constants in Theorem \ref{thm: tightAlg} is crucial at each step of the proof.

\medskip

The organization of this paper is as follows: In Section 2 we briefly recall the definitions and the tools that we use for our results. In Section 3 we prove our main results. In Section 4 we discuss Problems (1) and  (3). Finally, in Section 5 we discuss open problems and remarks. 

\medskip
\noindent
{\it Acknowledgments}. Foremost, we would like to express our sincere gratitude to our mentor Professor Laura De Carli for her continuous support of our research.

The research of this project was done during the Summer 2023 REU program ``AMRPU @ FIU" that took place at the Department of Mathematics and Statistics, Florida International University, and was supported by the NSF (REU Site) grant DMS-2050971. In particular, support of J.G. came from the above grant. We acknowledge the collaboration of Mikhail Samoshin in the initial stage of this project and his participation in the technical report (available at go.fiu.edu/amrpu).

\section{Preliminaries}

\subsection{Frames}

We have used the excellent textbook \cite{Christensen} for the definitions and some of the results presented in this section.
	
Let  ${\mathcal F}=\{f_j\}_{j\in J}$ be a frame in a Hilbert space ${\mathbb{H}}$. 
We recall the definition of three important operators that are associated with  ${\mathcal F}$.

    \begin{enumerate}
        \item The synthesis operator $T: \ell^2(\N) \to \mathbb{H}$ is given by $$T\{c_j\}_{j \in J} = \sum_{j \in J}c_jf_j.$$

        \item The analysis operator $T^*: \mathbb{H}\to\ell^2(\N)$ is given by $$T^*f= \{\langle f, f_j\rangle\}_{j \in J}.$$

        \item The frame operator $S:\mathbb{H}\to \mathbb{H}$ is given by $S=TT^*$ or equivalently, $$Sf = \sum_{j \in J}\langle f, f_j\rangle f_j.$$
    \end{enumerate}
If $ \H=\R^n$, the operator $T$ is represented by a matrix whose columns are the vectors of the frame. 

Frames, like bases, can be used to represent the vectors in the space.  However, with the added redundancy, we have to choose canonical ways in which to do this. If  $\{f_j\}_{j \in J}$ is  a Parseval frame, the frame operator $S=I$ where $I$ is the identity on $\mathbb{H}$ so we obtain an exact representation $$f=\sum_{j \in J}\langle f, f_j\rangle f_j$$ for every $f\in \mathbb{H}$, see \cite[Section 3.2]{Han}. Although the representation is not unique and $\{f_j\}_{j \in J}$ need not be an orthonormal basis, we still can recover coefficients using the inner product. 

If  $\{f_j\}_{j \in J}$ is a tight frame with frame bound $A$, then $S=AI$ where $I$ is the identity on $\mathbb{H}$. If $ \H=\R^n$,  the rows of the synthesis matrix $T$ are orthogonal and have the same norm $A$. For every $f\in \mathbb{H},$ $$f=\sum_{j \in J}\frac{1}{A}\langle f, f_j\rangle f_j.
$$
\noindent If the frame is not tight,  this representation formula needs to be modified. 
  Let  $\{f_j\}_{j \in J}$  be a frame with frame operator $S$. Then according to \cite[Section 6.1]{Han}, $S$ is invertible, $\{S^{-1}f_j\}_{j \in J}$ is a frame called the canonical dual frame, and for every $f\in \mathbb{H}$, 
  \begin{equation}\label{repr-formula}
  	f= \sum_{j \in J}\langle f, S^{-1}f_j\rangle f_j= \sum_{j \in J}\langle f, f_j\rangle S^{-1}f_j.\end{equation}

 The following lemma establishes the connection between the norms of  the vectors in a tight frame and the bound $A$. 
	It appears as an exercise in \cite{Han}, 
	but we provide a proof for the convenience of the reader.

\begin{lemma}\label{lem: TightFrameBound}
	If $\{f_j\}_{j=1}^k$ is a tight frame for $\R^n$ with frame bound $A$ then $$A=\frac{1}{n} \sum_{j=1}^{k} ||f_j||^2.$$
\end{lemma}

\begin{proof}
	Let $\{e_i\}_{i=1}^n$ be the canonical orthonormal basis for $\R^n$. Then:
	
	\begin{align*}
		\frac{1}{n} \sum_{j=1}^k \norm{f_j}^2 
  &= \frac{1}{n} \sum_{j=1}^k \norm{ \sum_{i=1}^n \langle f_j,e_i \rangle e_i }^2
		= \frac{1}{n} \sum_{j=1}^k \sum_{i=1}^n \norm{ \langle f_j,e_i \rangle e_i }^2  \\
  &= \frac{1}{n} \sum_{j=1}^k \sum_{i=1}^n | \langle f_j,e_i \rangle |^2
  = \frac{1}{n} \sum_{i=1}^n \left( \sum_{j=1}^k | \langle e_i,f_j \rangle |^2 \right)\\
		&= \frac{1}{n} \sum_{i=1}^n \left( A \norm{e_i}^2 \right) = A.
	\end{align*}
	
\end{proof}
\subsection{The Paley-Wiener theorem} 

The Paley-Wiener theorem is one of the fundamental stability criteria. It was originally proved by R. Paley and N. Wiener in 1934 and it provides a sufficient condition for the stability of Riesz bases. A proof of this theorem can be found e.g. in \cite[Theorem 10]{Young}.

\begin{theorem}[Paley-Wiener]\label{thm:PWclassically}
    Let $\{f_j\}_{j \in J}$ be a Riesz basis for Banach space $X$ and let $\{g_j\}_{j \in J}$ be a sequence in $X$. Suppose there exists $0\le \lambda<1$ such that \begin{equation}
        \norm{\sum_{j \in J}c_j(f_j-g_j)}\le \lambda\norm{\sum_{j \in J}c_jf_j}
    \end{equation}
    for all finite sequences $\{c_j\}_{j \in J}$. Then, $\{g_j\}_{j \in J}$ is a Riesz basis for $X$.
\end{theorem}

\noindent O. Christensen proved the following extension of Theorem \ref{thm:PWclassically} to frames in \cite{Christensen}.
\begin{theorem}\label{thm: Paley-Wiener}
    Let $\{f_j\}_{j \in J}$ be a frame for $\mathbb{H}$ with frame bounds $A, B>0$ and let $\{g_k\}_{k \in J}$ be a sequence in $\mathbb{H}$. Suppose there exists $\lambda, \mu\ge 0$ such that $\lambda+\frac{\mu}{\sqrt{A}}<1$ and \begin{equation}\label{eq: PWIneq}
        \norm{\sum_{k \in J}c_k(f_k-g_k)}\le \lambda\norm{\sum_{k \in J}c_kf_k}+\mu\left(\sum_{k \in J}|c_k|^2\right)^{\frac{1}{2}}
    \end{equation}

    for all finite sequences $\{c_k\}_{k \in J}$. Then, $\{g_k\}_{k \in J}$ is a frame for $\mathbb{H}$ with frame bounds $$A\left(1-\lambda-\frac{\mu}{\sqrt{A}}\right)^2, \text{ and } \; B\left(1+\lambda+\frac{\mu}{\sqrt{B}}\right)^2.$$ 

\end{theorem}

\section{Proofs of Theorems \ref{thm: betterFrame} and \ref{thm: tightAlg}}
First, we prove a simple lemma on which the proof of Theorems \ref{thm: betterFrame} and \ref{thm: tightAlg} are based. 
  
\begin{lemma}\label{lem: newEigenvalues}
    Let $S$ be an invertible, self-adjoint $n\times n$ matrix. Then for every $r>0$, $S+2rI_{n\times n}+r^2S^{-1}$ has all eigenvalues of the form $\lambda+2r+\frac{r^2}{\lambda}$ where $\lambda$ is an eigenvalue of $S$.  
\end{lemma}

\begin{proof}
    Since $S$ is self-adjoint, by the Spectral theorem, $S$ is diagonalizable. Hence, there exists $Q$ invertible and $D= \text{diag}(A, \dots, B)$ where $S= QDQ^{-1}.$ Since $S$ is invertible, $D$ is invertible. Thus, $$S+2rI_{n\times n}+r^2S^{-1}= QDQ^{-1}+2rI_{n\times n}+r^2QD^{-1}Q^{-1}= Q(D+2rI_{n\times n}+r^2D^{-1})Q^{-1}.$$ Hence, the eigenvalues of $S+2rI_{n\times n}+r^2S^{-1}$ are of the form $\lambda+2r+\frac{r^2}{\lambda}$ where $\lambda$ is an eigenvalue of $S$.
\end{proof}
\subsection{Proof of Theorem $\ref{thm: betterFrame}$}
\begin{proof}

    Let $T$ be the synthesis operator and $S=TT^*$ be the frame operator for $\{v_j\}_{j=1}^k$. By \cite[Proposition 3.19]{Han}, every $x\in\R^n$ can be written as $$x=\sum_{j=1}^k \langle x, S^{-1}v_j\rangle v_j.$$ Let $\delta_j = rS^{-1}v_j$ where $0<r< \min\left(\frac{\epsilon}{\max_{j}||S^{-1}v_j||}, A\right)$.
     Now, for $x\in \R^n$ we have:
    \begin{align*}
        \sum_{j=1}^k|\langle x, \wt{v}_j \rangle|^2 &= \sum_{j=1}^k|\langle x, v_j \rangle|^2+\sum_{j=1}^k 2\langle x,v_j\rangle\langle x, \delta_j\rangle +\sum_{j=1}^k|\langle x, \delta_j\rangle|^2\\
        &=\sum_{j=1}^k|\langle x, v_j \rangle|^2+\sum_{j=1}^k2\langle x,v_j\rangle\langle x, rS^{-1}v_j\rangle +\sum_{j=1}^k|\langle x, rS^{-1}v_j\rangle|^2\\
        &= \sum_{j=1}^k|\langle x, v_j \rangle|^2+2r\langle x, \sum_{j=1}^k\langle x,  S^{-1}v_j\rangle v_j\rangle +r^2\sum_{j=1}^k|\langle x, S^{-1}v_j\rangle|^2\\
        &=\langle Sx, x\rangle +2r\langle x, x\rangle +r^2\langle S^{-1}x, x\rangle \\
        &=\langle (S+2rI_{n\times n}+r^2S^{-1})x, x\rangle\\
        &= \langle (T+rS^{-1}T)(T+rS^{-1}T)^*x, x\rangle. 
    \end{align*}

    Let $\lambda_{min},\lambda_{max}$ be the minimum and maximum eigenvalues of $S$. By \cite[Proposition 3.27]{Han}, $\lambda_{min}>0$. By Lemma \ref{lem: newEigenvalues} and \cite[Proposition 3.27]{Han}, since $r<A\le \lambda_{min}$ and the map $\lambda\mapsto \lambda+2r+\frac{r^2}{\lambda}$ is increasing for $\lambda\ge r$, the optimal lower and upper frame constants are
    $$A_1=\lambda_{min}+2r+\frac{r^2}{\lambda_{min}}\ge A+2r+\frac{r^2}{A},$$ and  $$B_1=\lambda_{max}+2r+\frac{r^2}{\lambda_{max}}\le B+2r+\frac{r^2}{B}.$$

    \noindent Hence, $$\frac{B_1}{A_1}\le \frac{B+2r+\frac{r^2}{B}}{A+2r+\frac{r^2}{A}}.$$ Note finally $$\frac{B+2r+\frac{r^2}{B}}{A+2r+\frac{r^2}{A}}<\frac{B}{A} \iff 2r(B-A)+r^2\left(\frac{B}{A}-\frac{A}{B}\right)>0,$$ so the result follows.
\end{proof}

Using the approach in Theorem \ref{thm: betterFrame}, we have an algorithm which produces to a tight frame in finitely many steps. Each step depends on the frame bounds and the inverse of the frame operator of the previous step. 

\subsection{Proof of Theorem $\ref{thm: tightAlg}$}
\begin{proof}
    Let $T_0$ be the synthesis operator and $S_0=T_0T_0^*$ be the frame operator for $\{v_j\}_{j=1}^k$. Let $A_0=A$ and $B_0=B$.

    \noindent For each $1\le m\le n-1$, define $r_m$ and $\{v_j^{(m)}\}_{j=1}^k$ by, $$\begin{cases}
        r_m=\sqrt{A_{m-1}B_{m-1}}\\
        v_j^{(m)}= v_j^{(m-1)}+r_mS_{m-1}^{-1}v_j^{(m-1)} & j=1,2,\dots, k,
    \end{cases}$$ where $A_m, B_m$ are the frame bounds for $\{v_j^{(m)}\}_{j=1}^k$.

    Moreover, let the synthesis operator for $\{v_j^{(m)}\}_{j=1}^k$ be $$T_m = \begin{bmatrix}
            \vert & \vert & & \vert\\
            v_1^{(m)} & v_2^{(m)} & \vdots & v_k^{(m)}\\
            \vert & \vert & & \vert
    \end{bmatrix}$$ and the frame operator 
        $S_m = T_mT_m^*$.

    We will show via induction that for each $1\le m\le n-1$, $\{v_{j}^{(m)}\}_{j=1}^k$ is a frame and at least $m+1$ of the eigenvalues of $S_m$ are equal to $B_m$.  Indeed, for $m=1$, $r_1=\sqrt{AB}.$  Note that $T_1 = T_{0}+r_1S_{0}^{-1}T_{0}.$ Hence, \begin{align*}
        S_1 = T_1T_1^* &= (T_{0}+r_1S_{0}^{-1}T_{0})(T_{0}^*+r_1T_0^*S_{0}^{-1})\\
        &= S_0+r_1S_0^{-1}T_0T_0^*+ r_1T_0T_0^*S_0^{-1}+r_1^2S_0^{-1}T_0T_0^*S_0^{-1}\\
        &=S_0+2r_1I_{n\times n}+r_1^2S_0^{-1}.
    \end{align*}
    
     \noindent By Lemma \ref{lem: newEigenvalues}, the eigenvalues of $S_1$ are of the form $$\lambda^\prime =\lambda+2r_1+\frac{r_1^2}{\lambda},$$ where $\lambda$ ranges over the eigenvalues of $S_0$. Since $\lambda\ge A>0$, $\lambda^\prime >0$, and so $\{v_j^{(1)}\}_{j=1}^k$ is a frame. Moreover, two of the eigenvalues of $S_1$ are $$A+2r_1+\frac{r_1^2}{A}, \text{ and } \; B+2r_1+\frac{r_1^2}{B}$$ which are both equal and so, equal to $B_1$.

     Now suppose for some $1\le m\le n-2$, we have $\{v_j^{(m)}\}_{j=1}^k$ a frame, and $m+1$ of the eigenvalues of $S_m$ equal to $B_m$. Note $r_{m+1}=\sqrt{A_mB_m}$. By the same computation as in the base case, the eigenvalues of $S_{m+1}$ are of the form $$\lambda^\prime =\lambda+2r_{m+1}+\frac{r_{m+1}^2}{\lambda},$$ where $\lambda$ ranges over the eigenvalues of $S_m$. Since $\lambda\ge A_m>0$, we again have $\{v_j^{(m+1)}\}_{j=1}^k$ a frame. After the mapping $$\lambda \mapsto \lambda+2r_{m+1}+\frac{r_{m+1}^2}{\lambda},$$ the $m$ eigenvalues of $S_m$ equal to $B_m$ are still equal. Moreover, since $r_{m+1}=\sqrt{A_mB_m}$,  $$A_m+2r_{m+1}+\frac{r_{m+1}^2}{A_m}= B_m+2r_{m+1}+\frac{r_{m+1}^2}{B_m}=B_{m+1}.$$ Thus, we have $m+1$ eigenvalues of $S_{m+1}$ equal to $B_{m+1}$ (namely, those mapped from the $m$ equal on the previous step and the one mapped from $A_m$). This completes the induction.

    Therefore, $S_{n-1}$ is a multiple of the identity since it has all equal eigenvalues. Hence, $\{v_j^{(n-1)}\}_{j=1}^k$ is a tight frame. Then, let $\delta_j= v_j^{(n-1)}-v_j$ for each $1\le j\le k$.

    Finally, we show $\{v_j^{(n-1)}\}_{j=1}^k$ is nonzero. In particular, after every iteration, none of the frame elements become $0$. This follows from induction. Indeed, for $m=0$, all of the frame elements are nonzero. Suppose for some $m\ge 0$, $\{v_j^{(m)}\}_{j=1}^k$ is nonzero. Observe for every $1\le j\le k$,  $$v_j^{(m+1)}= (I_{n\times n}+r_{m+1}S_{m}^{-1})v_{j}^{(m)}.$$ It suffices to show $I_{n\times n}+r_{m+1}S_{m}^{-1}$ is nonsingular. By the diagonalization argument as in Lemma \ref{lem: newEigenvalues}, the eigenvalues of $I_{n\times n}+r_{m+1}S_{m}^{-1}$ are of the form $1+\frac{r_{m+1}}{\lambda}$ where $\lambda$ ranges over the eigenvalues of $S_m$. Since these are all positive, we are done.
\end{proof}

\section{Perturbations in Finite Dimensions }
In this section we focus on Problems (1) and (3) stated  in the introduction. First we consider Problem (1): "How much can a frame be perturbed and still remain a frame?".

Given a frame, we quantify additive perturbations that preserve the frame structure. We begin by establishing a sharp bound on the norm of the perturbations.

\begin{prop}\label{thm: StabilitywithPW}
    Suppose $\{v_j\}_{j=1}^k$ is a frame in $\R^n$ with frame bounds $A, B>0$. If $\epsilon\le \frac{\sqrt{A}}{\sqrt{k}}$, then for every $\{\delta_j\}_{j=1}^k$ with $||\delta_j||< \epsilon$ ,  the set $\{v_j+\delta_j\}_{j=1}^k$ is a frame.
\end{prop}

\begin{proof}

Fix $\{\delta_j\}_{j=1}^k$ with $||\delta_j||< \epsilon$. Let $\alpha = \max_{j}||\delta_j||$. We apply Theorem \ref{thm: Paley-Wiener}. Fix $\{c_m\}_{j=1}^k\subset \R$. Then by the triangle inequality and Cauchy-Schwarz, \begin{align*}
    \left|\left|\sum_{j=1}^k c_j\delta_{j}\right|\right| &\le \sum_{j=1}^k |c_j|\cdot ||\delta_j||\le \left(\sum_{j=1}^k |c_j|^{2}\right)^{\frac{1}{2}}\left(\sum_{j=1}^k ||\delta_j||^2\right)^{\frac{1}{2}}\\
    &\le\alpha\sqrt{k}\left(\sum_{j=1}^k |c_j|^{2}\right)^{\frac{1}{2}}.
\end{align*}
Since $\frac{\alpha\sqrt{k}}{\sqrt{A}}< \frac{\epsilon\sqrt{k}}{\sqrt{A}}\le 1$, we are done.
\end{proof}

There are frames for which this bound on $\epsilon$ is sharp. To show this, we have the following corollary:

\begin{corollary}\label{cor: UTFstability}
    Suppose $\{v_j\}_{j=1}^k$ is a unit tight frame in $\R^n$. If $\epsilon\le  \frac{1}{\sqrt{n}}$, then for every $\{\delta_j\}_{j=1}^k$ with $||\delta_j||< \epsilon$ , $\{v_j+\delta_j\}_{j=1}^k$ is a frame. This bound is optimal for $\epsilon$.
\end{corollary}

\begin{proof}
    Note by Lemma \ref{lem: TightFrameBound}, $A=\frac{k}{n}$ for a unit tight frame. Then by Proposition \ref{thm: StabilitywithPW}, we reach the desired bound on $\epsilon.$
    We will now give an example of equality. Let the frame be $\left\{\begin{bmatrix}
    1\\ 0
\end{bmatrix},\begin{bmatrix}
    0\\ 1
\end{bmatrix}\right\}$ . By \cite[Lemma 4.1]{Han}, this is a unit-tight frame. Moreover, if $\epsilon=\frac{1}{\sqrt{2}}$, then we can perturb all the frame elements to be $\frac{1}{2}\begin{bmatrix}
    1\\ 1
\end{bmatrix}.$ Clearly, this is no longer a frame.

\end{proof}

\begin{remark}
Our Corollary \ref{cor: UTFstability} is a more general version of the following problem that appears as exercise 6.B.14 in \cite{Axler}:
\begin{center}
   Suppose $\{e_j\}_{j=1}^n$ is an orthonormal basis of $\R^n$ and $\{v_j\}_{j=1}^n$ such that $\norm{e_j-v_j}\le \frac{1}{\sqrt{n}}$ for each $j$. Prove $\{v_j\}_{j=1}^n$ is a basis.
\end{center}
\end{remark}
 
\subsection{Adding Vectors to a Frame}
In this section, we focus our attention on Problem (3) stated in the introduction with regard to tight frames. Theorem 3.3 in \cite{LS2009} establishes that if a tight frame is augmented by $N$ vectors and the resulting frame is still tight, then $N$ must be greater than the dimension of the space.

In the case of $\R^n$, we show that the only way that a tight frame can be augmented and still remain tight is when the vectors added form a tight frame themselves. Note that the vectors added forming a tight frame implies that the amount of vectors needs to be at least $n$, the dimension of $\R^n$, since being a frame in finite dimension is equivalent to being a spanning set.

\begin{prop}
	\label{TightFramesProp}
	Let $\{v_i\}_{i=1}^k$ be a tight frame for $\R^n$ with frame bound $A$. Append $p$ vectors $v_{k+1},v_{k+2},...,v_{k+p}$ obtaining $\{v_i\}_{i=1}^{k+p}$. Then, $\{v_i\}_{i=1}^{k+p}$ is a tight frame for $\R^n$ if and only if  $\{v_i\}_{i=k+1}^{k+p}$ is a tight frame for $\R^n$.
\end{prop}

\begin{proof}
The reverse direction follows directly from the definition of a tight frame. It suffices to show the forward direction. Since $\{v_i\}_{i=1}^k$ is tight with bound $A$ we have $\sum_{i=1}^{k} |\langle v,v_i \rangle|^2 = A||v||^2$ for all $v \in \R^n$. Suppose $\{v_i\}_{i=1}^{k+p}$ is tight with frame bound $\hat{A}$ so that $\sum_{i=1}^{k+p} |\langle v,v_i \rangle|^2 = \hat{A}||v||^2$ for all $v \in \R^n$. Then for all $v \in \R^n$: 
\begin{gather*}
\hat{A}||v||^2 = \sum_{i=1}^{k+p} |\langle v,v_i \rangle|^2 = \sum_{i=1}^{k} |\langle v,v_i \rangle|^2 + \sum_{i=k+1}^{k+p} |\langle v,v_i \rangle|^2 = A||v||^2 + \sum_{i=k+1}^{k+p} |\langle v,v_i \rangle|^2 \\
\implies \sum_{i=k+1}^{k+p} |\langle v,v_i \rangle|^2 = \left( \hat{A}-A \right)||v||^2
\end{gather*}
which shows that $\{v_i\}_{i=k+1}^{k+p}$ is tight with frame bound $(\hat{A}-A)$.

\end{proof}

It is known that every finite frame for $\R^n$ can be turned into a tight frame with the addition of at most $n-1$ vectors. See Proposition 6.1 in \cite{Casazza2016}. If you start with a tight frame nothing needs to be added to make it tight and of course, adding zero vectors (not adding a vector at all) is in accordance with the "at most $n-1$ vectors" condition. However, it is interesting to note that \ref{TightFramesProp} implies that if the frame is already tight and vectors will be added, we need $n$ or more vectors to produce another tight frame. 

Applications making use of frames require the associated algorithms to be numerically stable, which tight frames satisfy optimally as noted in \cite{KUTYNIOK2013}. Thus, a key question in frame theory is the following: given a frame, can the frame vectors be modified to become a tight frame? One way to do so is by scaling each frame vector as in \cite{DeCarliCassaza}. This motivates the following definition for a scalable frame which we borrow from \cite{DeCarliCassaza}:

\begin{definition}
	A frame $\left\{x_i\right\}_{i=1}^{m}$ for $\R^n$ is scalable if there exist non-negative constants $\left\{a_i\right\}_{i=1}^{m}$ for which $\left\{a_i x_i\right\}_{i=1}^{m}$ is a tight frame.
\end{definition}

Our Proposition \ref{TightFramesProp} immediately produces the following result:

\begin{corollary}
	If $\{v_i\}_{i=1}^k$ is a frame for $\R^n$ which contains a scalable sub-frame with $p $ elements; if $k-p<n$, then $\{v_i\}_{i=1}^k$ itself is not scalable.
\end{corollary}

Although Proposition \ref{TightFramesProp} is from the perspective of adding vectors to a tight frame, we obtain corollaries that extend results regarding erasures. Frames have gained significant traction due to their resilience to noise in signal processing applications but also due to their ability to withstand transmission losses which present themselves mathematically as erasures of frame elements.

The topic of erasures has been addressed in the past, see \cite{GOYAL2001} and \cite{HOLMES2004}. More recently, the situation where there is a single erasure has been analyzed in \cite{Datta2020}. Via our corollaries, we provide improvements to some results in the aforementioned paper. 

\begin{corollary}
	If $p=1$ so that one vector is being added to a tight frame, the resulting frame will never be tight unless $n=1$ (in which case we are in $\R^1$ where every frame is tight).
\end{corollary}

With this corollary we can establish the same result as Theorem 2.3 in \cite{Datta2020} which states that if a vector is removed from a tight frame, the resulting set is no longer a tight frame. However, we can improve upon this with the following:

\begin{corollary}
	If the amount of vectors appended, $p$, is less than the dimension of the space, $n$, that is $p<n$, then the resulting frame is never tight.
\end{corollary}

From this corollary, we can deduce that if we start with a tight frame for $\R^n$ and $p$ vectors are removed, as long as $p<n$ the resulting set is never a tight frame. We see that erasures in numbers up to the dimension of the space never result in a tight frame. This raises the following question: Is it possible to have a certain amount of erasures to a tight frame so that the resulting set is still a tight frame? There is no answer to this question with what is present in \cite{Datta2020} or anywhere else as far as I know. However, we can provide one with the following:

\begin{corollary}
	If $p \geq n$ then the resulting frame is tight only when the set of vectors added make a tight frame themselves.
\end{corollary}

Thus we answer in the affirmative: a tight frame can encounter erasures so that the resulting set is still a tight frame, but that occurs if and only if the removed vectors form a tight frame themselves.

We have provided a complete characterization of a tight frame encountering erasures and whether the resulting set after the erasures is a tight frame or not:

\begin{corollary} \hfill
	\begin{itemize}
		\item If a tight frame for $\R^n$ encounters $p < n$ erasures, the resulting set is never a tight frame.
		\item If a tight frame for $\R^n$ encounters $p \geq n$ erasures, the resulting set is a tight frame if and only if the set of vectors removed form a tight frame themselves.
	\end{itemize}
\end{corollary}

The condition number of a frame corresponds to the classical condition number of the frame operator matrix which is a measure of the effect of perturbations on the inverse.
A low condition number is desired for purposes of numerical stability in the reconstruction formula which makes use of the inverse of the frame operator matrix. Changes in the condition number after erasures is of interest and \cite{Datta2020} states in Proposition 2.5 that for a unit norm tight frame the condition number always increases after a single erasure. But we can now extend Proposition 2.5 and drop the requirement that the tight frame must be unit norm:

\begin{prop}
	The condition number of a tight frame always increases after a single erasure
\end{prop}

\begin{proof}
	A single erasure on a tight frame, which has condition number 1, leaves the resulting set not being a tight frame with a condition number greater than 1.
\end{proof}

 This raises an important question: what is the new condition number of a tight frame that has had some frame elements removed? Proposition 2.10 in \cite{Datta2020} establishes a worst case and best case of the new condition number after a single erasure but it is worthwhile to explore this problem further in more generality.

\section{Open problems and Remarks}
The representation formula \eqref{repr-formula}  requires the inversion of the frame operator matrix. The optimal scenario from a computational perspective is encountered with tight frames in which the frame operator is a multiple of the identity and hence the easiest to invert. The next best case would be having a diagonal frame operator as they are also easy to invert and thus offer benefits similar to tight frames. In \cite{Datta2017} the following question is posed as a possible research direction: Given a frame $\left\{ v_i \right\}$, what are the operators $M$ for which $\left\{M v_i\right\}$ is a frame whose frame operator is a diagonal matrix? This question can also be approached from the perspective of additive perturbations: Given a frame $\left\{ v_i \right\}$, what are the perturbations $\left\{ d_i \right\}$ for which $\left\{v_i + d_i\right\}$ is a frame whose frame operator is a diagonal matrix?

We show a simple example of a perturbation that transforms a frame in $\R^2$ into a frame with a diagonal frame operator. 

We denote with   $e_1=(1, 0), e_2=(0,1)$ the vectors of the canonical basis of $\R^2$.

\begin{prop}\label{prop:diagR2}
 If $\{v_j\}_{j=1}^k$ is a nonzero frame in $\R^2$, then there exists $\epsilon\in \R$ such that the frame operator of $\{v_j\}_{j=1}^{k-1}\cup \{v_k+\epsilon e_j\}$ is diagonal for some $j\in\{1, 2\}$. 
\end{prop}

\begin{proof}
	Let $T=\begin{bmatrix}
		u_1 \\ u_2
	\end{bmatrix}$ be the synthesis operator for $\{v_j\}_{j=1}^k$ 
Here  $u_1, u_2$ are the rows of the matrix $T $.
    Let $v_s(i)$ be the entry in $T$ with largest absolute value i.e. maximum $|v_j(i)|$. Reorder the frame elements where this is the last vector $v_k$.
    	
    	Let $$\epsilon= -\frac{\langle u_1, u_2\rangle}{v_k(i)} $$
    and let  $T^\prime =\begin{bmatrix}
        u_1^\prime \\ u_2^\prime
    \end{bmatrix}$ be the synthesis operator for $\{v_j\}_{j=1}^{k-1}\cup \{v_k+\epsilon e_i\}$. Thus, \begin{align*}
        \langle u_1^\prime, u_2^\prime \rangle &= \langle u_1, u_2\rangle + \epsilon v_k(i)=0.
    \end{align*}

    \noindent Hence, the frame operator $T^\prime(T^\prime)^*$ is diagonal.
\end{proof}

With this generalization, we can rephrase one of the questions that we posed in the introduction regarding frame stability in terms of diagonal frames:

\begin{enumerate}
    \item[(2*)] How much does a frame need to be perturbed to become a frame with a diagonal frame operator? 
\end{enumerate}

\subsection{Infinite Dimensional Frames}
To conclude this paper, we give some remarks about extending our main results to infinite dimensional Hilbert spaces.

The construction given in Theorem \ref{thm: betterFrame} relies on the following fact in finite dimensions: $$\max_{j \in \{1,...,k\}} \norm{S^{-1}v_j}<\infty.$$ In other words, the canonical dual frame is always bounded. However, in infinite dimensional spaces, it is not immediately obvious whether this is true. 

The construction given in Theorem \ref{thm: tightAlg} needs the frame operator to have finitely many eigenvalues. In infinite dimensions, the theorem holds for frames with frame operators having finite spectrum. However, most frames operators do not have this property, so more powerful tools will need to be used.

As mentioned in the introduction, a corollary to the proof of the Krein-Milman-Rutman theorem forbids an infinite dimension analog to Proposition \ref{thm: StabilitywithPW}. We leave as an open problem to characterize the  additive perturbations  $\{\epsilon_n\}_{n \in \N}$ in the Krein-Milman-Rutman theorem   that preserve frames. 

\medskip
We state a simple convexity property of perturbed frames which could allude to thinking of frames topologically:

\begin{prop}\label{prop:convexExtension}
    Suppose $\{x_n\}_{n \in \N}$ is a frame in $\mathcal{H}$ with frame bounds $A,B$. Moreover, suppose $\{y_n\}_{n \in \N}$ is a sequence in $\mathbb{H}$ such that $\{x_n\}_{n \in \N}$ and $\{y_n\}_{n \in \N}$ satisfy \eqref{eq: PWIneq} with $\lambda=0$ and $0\le \mu<\sqrt{A}$. Let $\tau < \frac{\sqrt{A}}{\mu}$. Then, for every sequence of complex numbers $\{t_n\}_{n \in \N}$ with $|t_n|\le \tau$, $$\left\{(1-t_n)x_n+t_ny_n\right\}_{n \in \N}$$ is a frame.
\end{prop}

\begin{proof}
    We apply Theorem \ref{thm: Paley-Wiener}. Fix $\{c_n\}_{n=1}^N$ in $\C$. Note $\sup_{n \in \N}|t_n|\le \tau$. Then, \begin{align*}
        \norm{\sum_{n \in \N} c_n(x_n - (1-t_n)x_n-t_ny_n) } &=\norm{\sum_{n \in \N}c_nt_n(x_n-y_n)}\\
        &\le \mu \left(\sum_{n \in \N} |c_nt_n|^2\right)^{\frac{1}{2}}\\
        &\le \mu \tau  \left(\sum_{n \in \N} |c_n|^2\right)^{\frac{1}{2}}.
    \end{align*}

    Since $\frac{\mu\tau}{\sqrt{A}}<1$, we are done.
\end{proof}

\begin{remark}
    Observe $\frac{\sqrt{A}}{\mu}>1$ so $\tau =1$ always satisfies Proposition \ref{prop:convexExtension}.
\end{remark}

\bibliographystyle{plain}
\bibliography{Main}

\end{document}